%% file: main.tex
\documentclass{amsart}

\input{macros}

\begin{document}


\title{Isosingular loci of algebraic varieties}

\author{Christopher Chiu}
  \address[C.\ Chiu]{%
    Fakult\"at f\"ur Mathematik\\
    Universit\"at Wien\\
    Oskar-Morgenstern-Platz 1\\
    A-1090 Wien (\"Osterreich)%
  }
  \email{christopher.heng.chiu@univie.ac.at}
  
\author{Herwig Hauser}
  \address[H.\ Hauser]{%
    Fakult\"at f\"ur Mathematik\\
    Universit\"at Wien\\
    Oskar-Morgenstern-Platz 1\\
    A-1090 Wien (\"Osterreich)%
  }
  \email{herwig.hauser@univie.ac.at}
  
\date{\today}


\thanks{The research of the first and the second author were supported by project P-31338 of the Austrian Science Fund (FWF)}
  
\begin{abstract}
 We define the notion of isosingular loci of algebraic varieties, following the analytic case first studied by Ephraim. In particular, we give a partial extension of his main result in arbitrary characteristic and a full extension assuming characteristic $0$. One of the main obstructions in the positive characteristic case is the non-separability of the orbit map associated to the contact group, as first observed by Greuel and Pham for isolated singularities.
\end{abstract}

\maketitle


Let $X$ be a complex analytic space and $p \in X$. In \cite{Eph78} the \emph{isosingular locus} of $X$ at $p$ was defined to be the subset of $X$ with a prescribed singularity type, that is,
\[
 \iso(X,p) = \{ q \in X \mid X_q \isom X_p\},
\]
where $X_p$ and $X_q$ denote the analytic germs of $X$ at $p$ and $q$. The main result of \cite{Eph78} asserts the following:

\begin{theorem*}[\protect{\cite[Theorem 0.2, Observation 2.5]{Eph78}}]
 For a complex analytic space $X$ and $p\in X$ the set $\iso(X,p)$ is locally closed and smooth as a (reduced) analytic subspace of $X$. Furthermore, there exists a germ of an analytic space $Y_y$ such that $X_p \isom Y_y \times \iso(X,p)_p$, with $Y_y$ having the additional property that there does not exist an isomorphism of germs $Y_y \isom Y'_{y'} \times \C_0$.
\end{theorem*}

Generalizations of this result to the relative case of morphisms between analytic spaces were obtained in \cite{HM89,HM88}. We call an analytic germ $(X,p)$ \emph{harmonic} if $\iso(X,p) = \iso (\Sing X,p)$. In \cite{GH85} the classical Mather--Yau theorem (see \cite{MY82}) was extended from the isolated singularity case to general harmonic singularities. Results in a very similar spirit appeared recently in the preprint \cite{ES21}.

In this paper, our main goal is to study isosingular loci of an algebraic variety $X$ over an arbitrary algebraically closed field $k$. As we do not have the analytic topology at our disposal, we will replace it by considering the corresponding \emph{formal neighborhoods} instead. That is, for any $x\in X(k)$ the \emph{isosingular locus} of $X$ at $x$ is defined to be the set
\[
 \iso(X,x):=\{x'\in X \mid \^X_{x'} \isom \^X_x\},
\]
where $\^X_x$ denotes the formal completion of $X$ along the singleton $\{x\}$. Note that any isomorphism of such formal completions induces an isomorphism of residue fields and thus $\iso(X,x) \subset X(k)$. Our first main result is an extension of the first part of the above theorem in \cite{Eph78}.

\begin{introtheorem}
 \label{t:thm-a}
 Let $X$ be a variety over an algebraically closed field $k$.
 \begin{enumerate}
     \item For each $x \in X(k)$ the subset $\iso(X,x)$ is locally closed in $X(k)$ (endowed with the Zariski topology).
     \item Denote by $X^{(x)}$ the unique reduced subscheme of $X$ whose $k$-points agree with $\iso(X,x)$. Then $X^{(x)}$ is smooth.
 \end{enumerate}
\end{introtheorem}

The proof follows the analytic case and uses the action of the contact group $\cK$ as well as a version of Artin's celebrated approximation result in \cite{Art69}. As for remaining part of \cite[Theorem 0.2]{Eph78}, assuming characteristic $0$ we obtain the following result:

\begin{introtheorem}
 \label{t:thm-b}
 Let $X$ be a variety over an algebraically closed field of characteristic $0$. For each $x\in X(k)$ there exists a scheme $Y$ of finite type over $k$, a point $y\in Y(k)$ and an isomorphism
 \[
  \^X_x \isom \^{(X^{(x)})}_x \times \^Y_y,
 \]
 such that $\^Y_y$ has \emph{no smooth factors}, that is, there does not exist an isomorphism $\^Y_y \isom \^Z_z \times \^{(\A^1)}_0$.
\end{introtheorem}

The assertion of \cref{t:thm-b} fails in positive characteristics, as can be seen in the example of the Whitney umbrella $f = x^p + y^p z$ in characteristic $p>0$, c.f.\ \cref{ex:whitney-char-p}. Let us mention that the proof of \cref{t:thm-b} crucially relies on the characteristic $0$ assumption in two separate places. First, it makes use of the Nagata--Zariski-Lipman criterion (see \cref{t:zariski}), which expresses the existence of smooth factors in terms of the existence of \emph{regular} derivations. This part of the argument could potentially be extended to positive characteristics by means of \emph{Hasse--Schmidt derivations}. The second important assumption which holds only in characteristic $0$ is that the orbit map associated to the (truncated) contact group $\cK_\b$ is separable for all $\b$ large enough. In fact, the fact that this fails in positive characteristic was first exhibited by Greuel and Pham in \cite{GP18,GP20} for certain isolated singularities. As we will see in \cref{ex:whitney-char-p-ii}, the results of \cref{s:thm-b} give new examples of this phenomenon, including the singularity $f = x^p +y^p z$ mentioned above. Moreover, it prompts the following question:

\begin{introquestion}
 Let $k$ be an algebraically closed field and $0 \in X\subset \A^N$ a hypersurface singularity given by $f \in k[x_1,\ldots,x_N]$. Assume that there exists $\b_0 > 0$ such that the orbit map $\cK_\b \to o(f_\b)$ is separable for $\b \geq \b_0$ (see \cref{s:thm-b}). Does the conclusion of \cref{t:thm-b} hold in this case?
\end{introquestion}

Note that we restrict the question to the hypersurface case as even then we do not know of any counterexample.

\subsection*{Conventions} Throughout this paper we always assume $X$ to be an algebraic variety, that is, a separated scheme of finite type over an algebraically closed field $k$. For any point $x \in X(k)$ we denote by $\^X_x$ the formal neighborhood of $X$ at $x$. That is, $\^X_x$ is given by the formal completion of $X$ along the closed subscheme $\{x\}$ and is thus isomorphic to $\Spf(\^{\cO_{X,x}})$.

%
%

\section{Contact equivalence and proof of \cref{t:thm-a}}
\label{s:thm-a}

Let $k$ be an algebraically closed field. We start by repeating the main definition: for any point $x\in X(k)$ the \emph{isosingular locus} of $X$ at $x$ is defined to be the set
\[
    \iso(X,x) := \{x'\in X \mid \^X_x \isom \^X_{x'} \}.
\]
If $X$ is smooth at $x$, then $\iso(X,x)$ is clearly an open subset of $X(k)$.

\begin{remark}
    Let us briefly discuss the situation when $k$ is not algebraically closed. Assume $\charK k = 0$ and consider the Whitney umbrella $X$ defined by $x^2+y^2z$ in $\A^3_k$. Clearly $\Sing X$ is just given by the $z$-axis. Taking any point of the form $x = (0,0,t)$ for $t\neq 0$, we see that the formal completion $\^X_x$ is isomorphic to the union of two hyperplanes if and only if $\sqrt{t} \in k$. Thus the isosingular loci of $X$ might not even be constructible as subsets of $X(k)$.
\end{remark}

Since both \cref{t:thm-a,t:thm-b} are local statements on $X$, we may assume for the rest of this paper that $X \subset \A^N$ is affine, given by polynomials $f_1,\ldots,f_n \in k[x_1,\ldots,x_N]$. To simplify notation, let us write $\xx = (x_1,\ldots,x_N)$ and $f = (f_1,\ldots,f_n)$.

Let us start by recalling a general version of the formal inverse function theorem:

\begin{lemma}
    \label{l:inv-fct-thm}
    Let $R$ be any ring and $S = R[[\xx]]$ or $S = R[\xx]/(\xx)^\b$ for some $\b > 0$. Let $\vp: S \to S$ be a map of $R$-algebras given by $\vp_i := \vp(x_i)$ and satisfying $\vp_i(0) = 0$. Then $\vp$ is an isomorphism if and only if $\det (\frac{\partial \vp_i}{\partial x_j}(0))_{i,j\leq N} \in R^*$. 
\end{lemma}

\begin{proof}
    Follows from considering the maps on the associated graded rings.
\end{proof}

To describe the action of the contact group we fix some notation. All schemes are considered relative to the base field $k$; in particular, $R$ will always denote a $k$-algebra. By $\cO_N$ we denote the (infinite-dimensional) scheme whose $R$-points are formal series in $R[[\xx]]$. Similarly, for $\b \geq 0$ we define the (finite-dimensional) scheme $\cO_{N,\b}$ via $\cO_{N,\b}(R) := R[\xx]/(\xx)^{\b+1}$. Clearly we have $\cO_N = \varprojlim_\b \cO_{N,\b}$ via the natural maps.

For any ring $R$ we consider the group $\Aut_R(R[[\xx]])$ of local $R$-automorphisms. By \cref{l:inv-fct-thm} we have
\begin{equation}
    \label{eq:aut-grp}
    \Aut_R(R[[\xx]]) = \{\vp \in R[[\xx]]^N \mid \vp(0) = 0,\: \det (\frac{\partial \vp_i}{\partial x_j}(0))_{i,j\leq N} \in R^*\}.
\end{equation}
Thus the assignment $R \mapsto \Aut_R(R[[\xx]])$ defines a group scheme $\Aut \cO_N$. Similarly, for each $\b \geq 0$ we obtain an algebraic group $\Aut \cO_{N,\b}$ whose $R$-points are elements of $\Aut_R(R[\xx]/(\xx)^{\b+1})$.

For any $n>0$ the group scheme $\Aut \cO_N$ acts on $\cO_N^n$ as follows: given $f \in R[[\xx]]^n$ and $\vp \in \Aut_R(R[[\xx]])$, we define $\vp \cdot f := f \circ \vp = f(\vp_1,\ldots,\vp_N)$. For $\b \geq 0$ an action of $\Aut \cO_{N,\b}$ on $\cO_{N,\b}^n$ is defined in the same way. These actions are compatible via the diagram
\begin{equation}
    \label{eq:diagram}
    \xymatrix{\Aut \cO_N \times \cO_N^n \ar[r] \ar[d] & \cO_N^n \ar[d] \\
            \Aut \cO_{N,\b} \times \cO_{N,\b}^n \ar[r] & \cO_{N,\b}^n.} 
\end{equation}

\begin{definition}
    For $N,n >0$ the \emph{contact group} $\cK = \cK_{N,n}$ is defined as the group scheme $\GL_n(\cO_N) \rtimes \Aut \cO_N$. That is, for each $R$ we have
    \[
        \cK (R) = \GL_n(R[[\xx]]) \rtimes \Aut_R(R[[\xx]]),
    \]
    where the semidirect product is taken with respect to the group homomorphism $\Aut_R(R[[\xx]]) \to \Aut \GL_n(R[[\xx]])$ given by
    \[
        \vp \mapsto (M = (m_{i,j})_{i,j} \mapsto M\circ \vp = (m_{i,j}(\vp_1,\ldots,\vp_N))_{i,j}).
    \]
    Similarly, for $\b\geq0$ the \emph{truncated contact group} $\cK_\b$ is defined as the algebraic group $\GL_n(\cO_{N,\b}) \rtimes \Aut \cO_{N,\b}$ and we have $\cK = \varprojlim_\b \cK_\b$.
\end{definition}

The action of $\Aut \cO_N$ on $\cO^n_N$ together with the natural action of $\GL_n(\cO_N)$ defines an action $\r: \cK \times \cO_N^n \to \cO_N^n$ which is given by
\[
    \r: (M,\vp;f) \mapsto M\cdot (f\circ \vp).
\]
Similarly we obtain an action $\r_\b$ of $\cK_\b$ on $\cO_{N,\b}^n$ and \cref{eq:diagram} implies that $\r$ and $\r_\b$ are compatible via the diagram
\begin{equation}
    \label{eq:contact-grp-compatible}
    \xymatrix{ \cK \times \cO^n_N \ar[r]^-\r \ar[d] & \cO^n_N \ar[d] \\
            \cK_\b \times \cO_{N,\b}^n \ar[r]^-{\r_\b} & \cO_{N,\b}^n.}    
\end{equation}
As a first remark we note that $\cK_\b$ is smooth independent of the characteristic of the base field $k$.

\begin{lemma}
    \label{l:smoothness}
    For each $\b\geq0$ the truncated contact group $\cK_\b$ is smooth over $k$. Moreover, for any $f \in \cO_{N,\b}(k)$ the map $(\r_\b)_f: \cK_\b \to \cO_{N,\b}^n$ given by $(M,\vp) \mapsto M\cdot (f\circ \vp)$ has locally closed image $o(f)$, which is smooth when considered as a reduced subscheme of $\cO_{N,\b}^n$ (we call $o(f)$ the \emph{orbit} of $f$).
\end{lemma}

\begin{proof}
    For the first assertion note that the underlying space of $\cK_\b$ is $\GL_n(\cO_{N,\b}) \times \Aut \cO_{N,\b}$ and that we have open immersions $\GL_n(\cO_{N,\b}) \inj \A^{n^2 M}$ and $\Aut \cO_{N,\b} \inj \A^{N M}$, where $M := \binom{N+\b}{N}$. The second assertion then follows from \cite[Proposition 7.4]{Mil17}.
\end{proof}

The main idea behind the use of the contact group in the proof of \cref{t:thm-a} is that two points of $X$ have isomorphic formal neighborhoods if and only if the respective Taylor expansions of $f$ at those points lie in the same orbit of $\cK$. We want to review this classical argument in our setting.
First, let $\g: X \to \cO_N^n$ be the morphism defined as follows: for each $k$-algebra $R$ the map $\g(R)$ is given by
\[
    a \in X(R) \subset R^N \mapsto f(\xx+a) \in R[[\xx]].
\]
In an analogous way we obtain morphisms $\g_\b: X \to \cO_{N,\b}^n$ for $\b \geq 0$.

\begin{lemma}
    \label{l:formal-eq-contact-eq}
    Let $x_1,x_2 \in X(k)$. Then $\^X_{x_1} \isom \^X_{x_2}$ if and only if $\g(x_2) \in o(\g(x_1))$.
\end{lemma}

\begin{proof}
    We may assume that $x_1 = 0$ and $x_2 = a \in k^N$. Then
    \[
        \^{\cO_{X,x_2}} \isom k[[\xx]]/(f(\xx+a)).
    \]
    The existence of an isomorphism $\^{\cO_{X,x_1}} \isom \^{\cO_{X,x_2}}$ is equivalent to the existence of an isomorphism $\vp \in \Aut_k(k[[\xx]])$ such that the following equality of ideals holds:
    \[
        (\vp(f(\xx)))=(f(\xx+a)).
    \]
    To proceed we make use of a trick of Mather:
 
 \begin{lemma}
 \label{l:mather}
    Let $A,B\in k^{n\times n}$. Then there exists $C\in k^{n\times n}$ such that $C(1-AB)+B$ is invertible.
 \end{lemma}

 \begin{proof}
    Let $r=\rk(B)$ and choose a basis $\{e_i\}_i$ for $k^n$ such that $Be_1,\ldots,Be_r$ are linearly independent and $Be_i=0$ for $i>r$. Let $e'_{r+1},\ldots,e'_n$ be such that $Be_i,e'_j$ form a basis. Then $C$ is the matrix representing the linear map given by $e_i\mapsto 0$ for $i\leq r$ and $e_i\mapsto e'_i$ for $i>r$.
 \end{proof}
 
    Write $R = k[[\xx]]$ and $\fm = (\xx)$. The above equality of ideals implies the existence of $A,B \in R^{n \times n}$ with $A \cdot f(\vp(\xx)) = f(\xx+a)$ and $B \cdot f(\xx+a) = f(\vp(\xx))$. Then, by \cref{l:mather}, there exists $C\in k^{n\times n}$ with $D = C(1-AB)+B \in R^{n \times n}$ invertible modulo $\fm$, which implies that $D$ is invertible in $R^{n\times n}$. Clearly, $D \cdot f(\xx+a) = f(\vp(\xx))$.
\end{proof}

Note the contact group $\cK$ is not of finite type and neither is the scheme $\cO_N^n$. In order to be able to apply \cref{l:smoothness} for the action of the truncated contact group $\cK_\b$ we make use of a variant of Artin's celebrated approximation results, which is commonly referred to as \emph{Universal Strong Artin Approximation}.

\begin{theorem}[\protect{\cite[Theorem 6.1]{Art69}}]
    \label{t:uniform-strong-artin-approximation}
    Let $k$ be a field. For each tuple $(n,m,d,\a)$ of non-negative integers there exists a $\b\geq0$ satisfying the following property: let $\xx=(x_1,\ldots,x_n)$ and $\yy=(y_1,\ldots,y_m)$ be two sets of variables and $f=(f_1,\ldots,f_r) \in k[\xx,\yy]^r$ with $\deg(f_i)\leq d$. Assume there exist polynomials $\bar{y}=(\bar{y}_1,\ldots,\bar{y}_m) \in k[\xx]^m$ with
    \[
        f(\xx,\bar{y}(\xx)) \equiv 0 \mod (\xx)^\b.
    \]
    Then there exist algebraic power series $y=(y_1,\ldots,y_m)$ with \\  $f(\xx,y(\xx))=0$ and
    \[
        \bar{y} \equiv y(\xx) \mod (\xx)^\a.
    \]
\end{theorem}

\begin{corollary}
    \label{c:trunc-contact-eq}
    Let $x_1,x_2 \in X(k)$. Then there exists a $\b>0$ such that we have $\g(x_2) \in o(\g(x_1))$ if and only if $\g_\b(x_2) \in o(\g_\b(x_1))$.
\end{corollary}

\begin{proof}
    Assume $x_1 = 0$ and let $x_2$ be given by $a = (a_1,\ldots,a_n) \in k^n$. Consider the system of equations in variables $M = (M_{ij})$ and $\yy = (y_1,\ldots,y_N)$ over $k[\xx]$ given by
    \[
        M \cdot f(\yy) - f(\xx+a) = 0.
    \]
    The condition $\g_\b(x_2) \in o(\g_\b(x_1))$ is equivalent to the existence of $(\bar{M}(\xx),\bar{y}(\xx))$ such that $\det(\bar{M}(0))$, $\det(\frac{\partial \bar{y}_i}{\partial x_j}(0))\in k^*$. Take $\a \gg 0$. By \cref{t:uniform-strong-artin-approximation} there a $\b>0$ such that, for each solution $(\bar{M}(\xx),\bar{y}(\xx))$ of this system module $(\xx)^\b$, there exists a solution $(M(\xx),y(\xx))$ with $\bar{M}(\xx) - M(\xx), \bar{y}(\xx) - y(\xx) \in (\xx)^\a$. Thus, in particular $\det(M(0))$, $\det(\frac{\partial y_i}{\partial x_j}(0)) \in k^*$, which in turn implies that $\g(x_2) \in o(\g(x_1))$. The other direction follows from the diagram \cref{eq:contact-grp-compatible}.
\end{proof}

Thus we obtain the first assertation of \cref{t:thm-a}:

\begin{proposition}
\label{t:isosingular-loc-cl}
    Let $X$ be a variety over an algebraically closed field $k$ and let $x \in  X(k)$. Then $\iso(X,x)$ is locally closed as a subset of $X(k)$.
\end{proposition}

\begin{proof}
    Follows from \cref{l:smoothness,l:formal-eq-contact-eq,c:trunc-contact-eq}.
\end{proof}

\begin{remark}
    Let us mention that the isosingularity loci $\iso(X,x)$ do not give a stratification of $X$ in general, since $X$ might have infinitely many points with distinct singularities. Consider the classical example by Whitney \cite[Example 13.2]{Whi65}, which is $X \subset \A^3$ defined by $f = xy(x+y)(x+zy)$. For each point $x = (0,0,a) \in \Sing X$ the associated tangent cone is a union of four planes, which have a well-defined cross-ratio depending on $a$. As any formal isomorphism induces a linear isomorphism of tangent cones we see that $\iso(X,x) = \{x\}$.
\end{remark}

\cref{t:isosingular-loc-cl} leads us the make the following definition:

\begin{definition}
    Let $X$ be a scheme of finite type over an algebraically closed field $k$. For each $x \in X(k)$ we define $X^{(x)}$ to be the unique reduced subscheme of $X$ whose $k$-points equal $\iso(X,x)$ and call it the \emph{isosingularity scheme} associated to $x$.
\end{definition}

To finish the proof of \cref{t:thm-a} we aim to show that $X^{(x)}$ is smooth. As we will see in the next section, we cannot conclude directly using that the orbit $o(\g(x))$ is smooth. Our strategy is therefore to use generic smoothness to establish the existence of $x'\in \iso(X,x)$ such that $X^{(x)}$ is smooth at $x'$ and then extend the isomorphism $\^X_x \isom \^X_{x'}$ \'etale-locally. To that end, recall that an \emph{\'etale neighborhood} $(U,y)$ of $x \in X(k)$ is an \'etale morphism $u: U \to X$ and $y \in U(k)$ with $u(y) = x$. Artin's approximation results then imply the following corollary:

\begin{lemma}[\protect{\cite[Corollary 2.5]{Art69}}]
    \label{l:artin-etale-nbhd}
    Let $x \in X(k)$ and $x' \in \iso(X,x)$, that is, $\^X_x \isom \^X_{x'}$. Then there exists a common \'etale neighborhood $(U,y)$ of $x$ and $x'$, that is, a diagram of \'etale morphisms
    \[
        \xymatrix{ & U \ar[ld]_u \ar[rd]^{u'} & \\ X & & X} 
    \]
    and $y \in U(k)$ with $u(y)=x$ and $u'(y)=x'$.
\end{lemma}

\begin{lemma}
    \label{l:isosing-etale}
    Let $f: U \to X$ be \'etale and $y \in U(k)$ with $x = f(y) \in X(k)$. Then $f^{-1}(\iso(X,x))=\iso(U,y)$ and the restriction $U^{(y)} \to X^{(x)}$ is \'etale.
\end{lemma}

\begin{proof}
    The first assertion follows from the fact that, for $y' \in U(k)$ and $x' = f(y')$ the morphism $f$ induces an isomorphism on completions $\^U_{y'} \isom \^X_{x'}$. To see the second claim, consider the fiber diagram
    \[
        \xymatrix{U \times_X X^{(x)} \ar[r] \ar[d] & X^{(x)} \ar[d] \\
                    U \ar[r]_f & X.}
    \]
     As a base change of $f$ the morphism $U \times_X X^{(x)} \to X^{(x)}$ is \'etale again and in particular, since $X^{(x)}$ is reduced, so is $U \times_X X^{(x)}$. Thus $U^{(y)} \isom U\times_X X^{(x)}$ and we are done.
\end{proof}

\begin{proposition}
    \label{p:isosing-smooth}
     Let $X$ be a variety over an algebraically closed field $k$ and let $x \in X(k)$. Then $X^{(x)}$ is smooth over $k$.
\end{proposition}

\begin{proof}
    By definition $X^{(x)}$ is geometrically reduced and thus the subset of $k$-smooth points of $X^{(x)}$ is dense open (see \SPC{056V}). Thus there exists $x' \in X^{(x)}(k) = \iso(X,x)$ smooth over $k$. By \cref{l:artin-etale-nbhd} there exists a common \'etale neighborhood $(U,y)$ of $x$ and $x'$. Thus, by \cref{l:isosing-etale}, we have $\^{(X^{(x)})}_x \isom \^{(U^{(y)})}_{y} \isom \^{(X^{(x')})}_{x'}$. Clearly $X^{(x')} \isom X^{(x)}$, and thus $\^{(X^{(x)})}_x$ is formally smooth over $k$, which in turn implies that $X^{(x)}$ is smooth at $x$.
\end{proof}

%
%

\section{The proof of \cref{t:thm-b} and separability of the orbit map}
\label{s:thm-b}

This section is devoted to the proof of \cref{t:thm-b}, which will involve studying the orbit map $\cK \to o(f)$. As both sides are non-Noetherian schemes of infinite dimension, we first need to introduce the right notion of smoothness in this setting.

\begin{definition}
    Let $k$ be a ring and $(R,\fm)$ be a local $k$-algebra. We say that $R$ is \emph{formally smooth} over $k$ if for every $k$-algebra $C$ with nilpotent ideal $J \subset C$ and every diagram
    \[
        \xymatrix{R \ar[r]^-{\bar{\psi}} \ar@{-->}[rd]^-\psi & C/J \\
                  k \ar[u] \ar[r] & C \ar[u]}
    \]
    such that $\bar{\psi}(\fm^n) = 0$ in $C/J$ there exists a diagonal arrow $\psi: R \to C$ making the diagram commute.
\end{definition}

Note that this is equivalent to saying that $(R,\fm)$ considered as a topological ring with respect to its $\fm$-adic topology is formally smooth in the sense of \cite[(19.3.1)]{ega-iv-pt1}.

\begin{lemma}
\label{l:formally-smooth}
    Let $(R,\fm)$ and $(S,\fn)$ be local $k$-algebras with $R/\fm = S/\fn = k$.
    \begin{enumerate}
        \item Assume $R = \varinjlim_n R_n$ with $\{(R_n,\fm_n)\}_{n\in \N}$ a direct system of local $k$-algebras smooth over $k$. Then $R$ is formally smooth over $k$.
        \item Assume $\vp: R\to S$ is a local map and $R$, $S$ are formally smooth over $k$. If the induced cotangent map $T^*\vp: \fm/\fm^2 \to \fn/\fn^2$ is injective, then there exists a retraction $\^S \to \^R$ of the completion map $\^\vp: \^R \to \^S$.
    \end{enumerate}
\end{lemma}

Compare the second assertion with the fact that any submersion between smooth manifolds allows for a local section.

\begin{proof}
    Let us start by proving (1). Since $R_n$ is smooth over $k$, it is in particular formally smooth over $k$. By \cite[(19.5.4)]{ega-iv-pt1} this is equivalent to the natural map
    \[
        \Sym_k (\fm_n/\fm_n^2) \to \gr(R_n)
    \]
    being a bijection. The colimit of these maps is given by
    \[ 
        \Sym_k (\fm/\fm^2) \to \gr(R)
    \]
    and thus this map is a bijection, which, again by \cite[(19.5.4)]{ega-iv-pt1}, implies that $R$ is formally smooth over $k$.
        
    To prove (2), let $x_j \in S$, $j\in J$, be elements whose images form a basis for $\fn/\fn^2$ and such that for $I \subset J$ the images of $x_i$, $i\in I$, form a basis for the subspace $\fm/\fm^2$. Then the bijection $\Sym_k(\fm/\fm^2) \isom \gr(R)$ from before induces an isomorphism
    \[
        k[[x_i \mid i\in I]] := \varprojlim_n k[x_i \mid i\in I]/(x_i \mid i\in I)^n \to \^R
    \]
    and similarly for $\^S$. In particular, the map $\^\vp$ is given by the inclusion
    \[
        k[[x_i \mid i\in I]] \hookrightarrow k[[x_j \mid j\in J]],
    \]
    which has a obvious retraction defined by $x_j \mapsto 0$ for $j\in J\setminus I$.
\end{proof}

Let us now go back to the situation of \cref{t:thm-b}, that is, $X$ is a variety over an algebraically closed field $k$ and for $x\in X(k)$ we let $X^{(x)}$ be the associated isosingularity scheme. We keep the notation of the last section. Our main result will establish the existence of ``enough'' regular derivations on $\^{\cO_{X,x}}$. Recall that a derivation $d \in \Der_k(k[[\xx]])$ is called \emph{regular} if there exists $g\in k[[\xx]]$ with $d(g) \in k[[\xx]]^*$.

\begin{lemma}
 \label{l:tangent-vectors-to-derivations}
 Assume that there exists $\b_0>0$ such that the orbit map $\cK_\b \to o(f_\b)$ is separable for all $\b \geq \b_0$. Then, for every tangent vector $a \in T_x X^{(x)}$ there exists a regular derivation $d_a \in \Der_k(k[[\xx]])$ satisfying $d_a(f) \subset (f)$ and $d_a(x) = a$.
\end{lemma}

\begin{proof}
 Write $Z_\b := o(f_\b)$ and let $\b \geq \b_0$. By \cref{s:thm-a}, $\cK_\b$ and $Z_\b$ are nonsingular varieties over an algebraically closed field and thus the orbit map $\cK_\b \to Z_\b$ is generically smooth. Since it is $\cK_\b$-equivariant (and the action of $\cK_\b$ on both sides is obviously transitive) this implies that it is smooth everywhere. In particular, the tangent map $T_1 \cK_\b \to T_{f_\b} Z_\b$ is surjective for all $\b \geq \b_0$. Setting $Z := o(f) \subset \cO^n$, we get that $T_1 \cK \to T_f Z$ is surjective. As both $\cK$ and $Z$ are colimits of schemes smooth over $k$, by \cref{l:formally-smooth} the morphism of formal schemes
 \[
  \Phi: \^\cK_1 \to \^Z_f 
 \]
 admits a section $\~\Psi: \^Z_f \to \^\cK_1$. Write $\Psi$ for the composition of the map $\g: \^{(X^{(x)})}_x \to \^Z_f$ with $\~\Psi$; this gives a factorization of $\g$ as
 \[
  \xymatrix{\^{(X^{(x)})}_x \ar[r]^-\Psi & \^\cK_1 \ar[r]^-\Phi & \^Z_f.}
 \]
 Let us analyze what this factorization means for the tangent map of $\g$. We may assume for convenience's sake that $x=0$. To proceed we consider the functorial description of formal neighborhoods via \emph{test rings} $(A,\fm)$, that is, $A$ is a local $k$-algebra with $A/\fm = k$ and $\fm^n = 0$ for some $n$. Let $a=(a_1,\ldots,a_N)$ be an $A$-point of $\^{(X^{(x)})}_x$, that is, $a_i \in \fm$ and $f(a_1,\ldots,a_N)= 0$. The map $\Psi(A)$ is given by
 \[
  a \mapsto (M(a),\vp(a)),\: M(a) \in \GL_n(A[[\xx]]),\vp(a) \in \Aut_k(A[[\xx]]),
 \]
 where $M(a) \equiv 1 \mod \fm$ and $\vp(a) \equiv \xx \mod \fm$. Composing with $\Phi(A)$ gives $\g(A)$ and thus the identity
 \begin{equation}
     \label{eq:func-identity}
     f(\xx+a) = M(a) \cdot f(\vp(a)).
 \end{equation}
 To compute $T_0 \g$ we take $A = k[\eps]/(\eps^2)$ and let $a = \~a \eps$ with $\~a \in k^N$. Taking Taylor expansions on both sides of \eqref{eq:func-identity} gives
 \[
  f(\xx) + \frac{\partial f}{\partial x}(\xx) \cdot \~a \eps = (1 + \~M(\~a) \eps) (f(\xx) + \frac{\partial f}{\partial \xx}(\xx) \cdot \~\vp(\~a)\eps),
 \]
 with $\~M(\~a) \in k[[\xx]]^{n \times n}$ and $\~\vp(\~a) \in k[[\xx]]^n$. Simplifying yields
 \begin{equation}
  \label{eq:tangent-map}
  \frac{\partial f}{\partial x}(\xx) \cdot (\~a - \~\vp(\~a)) = \~M(\~a) \cdot f(\xx).
 \end{equation}
 Define $ d:= \frac{\partial}{\partial \xx} \cdot (\~a - \~\vp(\~a)) \in \Der_k(k[[\xx]])$. As $\xx+\~\vp(\~a)\in \Aut_k(k[[\xx]])$, it follows that $d(0) = \~a$ and by \eqref{eq:tangent-map} we have $d(f) \subset (f)$. 
\end{proof}

Now the proof of \cref{t:thm-b} follows from \cref{l:tangent-vectors-to-derivations} together with the following variant of the classical Nagata--Zariski--Lipman criterion:

\begin{theorem}
 \label{t:zariski}
 Assume that $k$ is of characteristic $0$ and $X$ is a variety over $k$. Let $X' \subset X$ be a subvariety which is smooth of dimension $m$ at a point $x\in X'(k)$. Assume that for a choice of local coordinates $x'_1,\ldots,x'_m$ for $X'$ at $x$ the associated derivations $dx'_1,\ldots,dx'_m$ lift to derivations on $X$. Then
 \[
  \^X_x \isom \^{(X')}_x \times \^Y_y,
 \]
 for some variety $Y$ and $y \in Y(k)$.
\end{theorem}

\begin{proof}
 See for example \cite[Theorem 30.1]{Mat89}.
\end{proof}

\begin{proof}[Proof of \cref{t:thm-b}]
 Assume $\charK k = 0$, then the orbit map $\cK_\b \to o(f_\b)$ is separable since it is dominant. Thus \cref{l:tangent-vectors-to-derivations,t:zariski} together imply that
 \[
   \^X_x \isom \^{(X^{(x)})}_x \times \^Y_y,
 \]
 for some variety $Y$ and $y \in Y(k)$. Using \cref{l:dim-isosing} below it follows that $\^Y_q$ itself has no smooth factors.
\end{proof}

\begin{lemma}
\label{l:dim-isosing}
 Let $X$ be a scheme of finite type over an algebraically closed field $k$ and let $x \in X(k)$. Assume that $\^X_x \isom \^Y_y \times \^{(\A^m)}_0$. Then $\dim_x X^{(x)} \geq m$.
\end{lemma}

\begin{proof}
 By \cref{l:artin-etale-nbhd} we can find a common \'etale neighborhood $U$ for $x \in X$ and $y' = (y,0) \in Y \times \A^m$. Clearly $\iso(Y \times \A^m, y') \supset \{y\} \times \A^m$ and thus we are done using \cref{l:isosing-etale}.
\end{proof}

We now want to discuss the existence of a decomposition as in \cref{t:thm-b} for $k$ of positive characteristic. As the following example shows, this fails in general.

\begin{example}
 \label{ex:whitney-char-p}
 Let $k$ be of characteristic $p > 0$ and $X$ be the Whitney umbrella given by $x^p + y^p z$ in $\A^3_k$. We claim that $\^X_0$ is isomorphic to $\^X_x$, where $x = (0,0,t)$ for some $t \neq 0$. As $k$ is algebraically closed, there exists $s \in k$ with $s^p = t$. Consider now the change of coordinates $\vp$ given by
 \[
  x \mapsto x + y s, \: y \mapsto y, \: z \mapsto z.
 \]
 Then $\vp(x^p + y^p z) = x^p + y^p(z+t)$ and thus $\^X_0 \isom \^X_x$. In particular, we have that $\iso(X,0)$ is just the $z$-axis. 
 
 We claim that $\^X_0$ has no smooth factors and sketch the argument here. By an extension of \cref{t:zariski} (see for example \cite[Exercise 30.1]{Mat89}) it is sufficient to show that there does not exist a regular continuous \emph{Hasse--Schmidt derivation} $D \in \Der_k^\infty(\^{\cO_{X,0}})$. That is, there does not exist a map $D: \^{\cO_{X,0}} \to \^{\cO_{X,0}}[[t]]$ of $\^{\cO_{X,0}}$-algebras which is continuous for the respective adic topologies and such that there exists an element $g \in \^{\cO_{X,0}}$ with $g(0) = 0$ and $D(g)$ invertible. To that avail, suppose such a map $D$ is given by
 \[
  D(x) = \sum_{i\geq0} \~x_i t^i, \:  D(y) = \sum_{i\geq0} \~y_i t^i, \: D(z) = \sum_{i\geq0} \~z_i t^i,
 \]
 satisfying $\~x_0 = x$ and so on. Applying $D$ to the equation $x^p+y^pz = 0$ yields a system of equations for $\~x_i, \~y_i, \~z_i$. For simplicity we will give them explicitly only in the case $p = 2$:
 \begin{align*}
    y^2 \~z_1 & = 0 \\
    \~x_1^2 + y^2 \~z_2 + \~y_1^2 z & = 0 \\
    \ldots & 
 \end{align*}
 From the first equation it follows that $\~z_1 = 0$ and from the second that $\~x_1(0) = 0$. Now suppose that $\~y_1$ is invertible. Then the second equation gives
 \[
  z = - \frac{y^2 \~z_2 + \~x_1^2}{\~y_1^2}.
 \]
 Note that the right hand side has order $\geq 2$, which gives a contradiction. Thus it follows that for any $g \in \^{\cO_{X,0}}$ with $g(0) = 0$ we have that $D(g)$ is not invertible.
\end{example}

One of the main assumptions in the proof of \cref{t:thm-b} was the separability of the orbit map $\cK_\b \to o(f_\b)$ for $\b \gg 0$. As observed already in \cite[Example 2.9]{GP18}, this fails in positive characteristics for general isolated singularities. The example provided there was the cusp singularity $f = y^2 + x^3$ for $\charK(k)=2$. While obviously not applicable to the case of isolated singularities, \cref{l:tangent-vectors-to-derivations} can be used to construct related examples where the separability of the orbit map fails.

\begin{example}
 \label{ex:whitney-char-p-ii}
 Let $k$ be algebraically closed with $\charK k =2$ and consider the deformation $\~f = x^2 +y^3 + z y^2 \in k[x,y,z]$ of the cusp singularity $f = x^2 + y^3$. Set $X = V(\~f) \subset \A^3$; we claim that $X^{(0)} = V(x,y) \isom \A^1$. If $t\neq0$ and $x = (0,0,t)$, then an isomorphism between $\^X_0$ and $\^X_x$ is given by the map
 \[
  x \mapsto x + y s, \: y \mapsto y, \: z \mapsto z,
 \]
 where $s\in k$ with $s^2 = t$. However, there does not exist $d\in \Der_k(k[[x,y,z]])$ satisfying $d(\~f) \subset (\~f)$ and $d(0)=(0,0,1)$, as can be verified with an argument similar to the one in \cref{ex:whitney-char-p}. Therefore, by \cref{l:tangent-vectors-to-derivations} the orbit map $\cK_\b \to o(\~f_\b)$ is inseparable for infinitely many $\b>0$.
 
 Note that the same argument also works for $f = x^p +y^pz$ and $\charK k =p$, as in \cref{ex:whitney-char-p}.
\end{example}

As mentioned in the introduction, these examples prompt the question whether inseparability of the orbit map is the main obstruction to extending \cref{t:thm-b} to positive characteristics. We expect a further investigation into this problem to shed more light on the formal structure of singularities for $\charK k =p$.

\printbibliography[
heading=bibintoc,
title={References}
]

\end{document}

%% file: macros.tex
\usepackage[utf8]{inputenc}

\usepackage{amsmath}
\usepackage{amsthm}
\usepackage{amssymb}
\usepackage{mathtools}
\usepackage{bm}		
\usepackage{mathdots}
\usepackage{framed}
\usepackage{hyperref}
\usepackage[capitalize]{cleveref}
\usepackage{array}
\usepackage[all,cmtip]{xy}

\usepackage{dsfont}

\usepackage{xcolor}

\usepackage{pdfpages}

\usepackage[pass]{geometry}

\def\N{{\mathbb N}}

\def\C{{\mathbb C}}

\def\A{{\mathbb A}}

\def\cK{\mathcal{K}}

\def\cO{\mathcal{O}}

\def\fm{\mathfrak{m}}
\def\fn{\mathfrak{n}}

\def\xx{\underline{x}}
\def\yy{\underline{y}}

\def\a{\alpha}
\def\b{\beta}
\def\g{\gamma}

\def\eps{\varepsilon}

\def\r{\rho}

\def\vp{\varphi}

\def\.{\cdot}
\let\circum\^
\def\^{\widehat}
\def\~{\widetilde}

\def\inj{\hookrightarrow}

\def\({\left(}
\def\){\right)}

\def\*{{}^*}

\renewcommand{\and}{ \ \ \text{ and } \ \ }

\DeclareMathOperator{\rk} {rk}

\DeclareMathOperator{\GL}  {GL}

\DeclareMathOperator{\Spf} {Spf}

\DeclareMathOperator{\Aut} {Aut}

\DeclareMathOperator{\Sing} {Sing}

\DeclareMathOperator{\Sym} {Sym}

\DeclareMathOperator{\gr} {gr}

\DeclareMathOperator{\Der}{Der}

\DeclareMathOperator{\charK}{char}

\def\isom{\simeq} 

\DeclareMathOperator{\iso}{Iso}

\theoremstyle{plain}

\newtheorem{introtheorem}{Theorem}

\newtheorem{introquestion}[introtheorem]{Question}
\crefname{introtheorem}{Theorem}{Theorems}

\newtheorem{theorem}{Theorem}[section]
\newtheorem*{theorem*}{Theorem}
\newtheorem{proposition}[theorem]{Proposition}
\newtheorem{lemma}[theorem]{Lemma}
\newtheorem{corollary}[theorem]{Corollary}

\theoremstyle{definition}

\newtheorem{definition}[theorem]{Definition}

\theoremstyle{remark}

\newtheorem{remark}[theorem]{Remark}
\newtheorem{example}[theorem]{Example}

\numberwithin{figure}{section}

\numberwithin{equation}{section}

\usepackage[
backend=biber,
style=numeric,
giveninits=true
]{biblatex}

\DefineBibliographyStrings{english}{%
	backrefpage = {cited on page},
	backrefpages = {cited on pages},
}

\DeclareSourcemap{
  \maps[datatype=bibtex]{
    \map{
        \step[fieldset=doi,null]
        \step[fieldset=url,null]
        \step[fieldset=isbn,null]
        \step[fieldset=issn,null]        
        \step[fieldset=note,null]
        \step[fieldset=pagetotal,null]
    }
  }
}

\newcommand{\SPC}[1]{\cite[\href{https://stacks.math.columbia.edu/tag/#1}{Tag #1}]{stacks-project}
}